\documentclass{tsp-short}

\usepackage{amsmath}
\usepackage{amssymb}
\usepackage{amsthm}

\usepackage{float}
\usepackage{graphicx}
\usepackage{xcolor}
\usepackage[colorlinks = true,
linkcolor = blue,
urlcolor = blue,
citecolor = blue,
anchorcolor = blue]{hyperref}

\newcommand{\rd}{\textcolor{red}}

\theoremstyle{plain}
\newtheorem{theorem}{Theorem}
\newtheorem*{theorem*}{Theorem}
\newtheorem{lemma}{Lemma}
\newtheorem{corollary}{Corollary}
\theoremstyle{definition}
\theoremstyle{remark}
\newtheorem{remark}{Remark}

\newcommand{\be}{\begin{equation}}
    \newcommand{\ee}{\end{equation}}
    
\newcommand{\bes}{\begin{equation*}}
    \newcommand{\ees}{\end{equation*}}
    
\newcommand{\ninf}{n\to\infty}
\newcommand{\as}{\text{a.s.}}
\newcommand{\weakto}{\overset{w}{\to}}

\usepackage{mathtools}

\DeclarePairedDelimiter\ceil{\lceil}{\rceil}
\DeclarePairedDelimiter\floor{\lfloor}{\rfloor}
\DeclarePairedDelimiter\p{(}{)}
\DeclarePairedDelimiter\event{\Big\{}{\Big\}}
\DeclarePairedDelimiter\Prob{\mathbb P\Big(}{\Big)}

\newcommand{\ve}{\varepsilon}

\newcommand{\mbE}{{\mathbb E}}

\newcommand{\mbP}{{\mathbb P}}

\newcommand{\Ce}{\mathrm{C}[0,1]}

\newcommand{\De}{\mathcal{D}[0,1]}

\newcommand{\supp}{\supp{\rm supp}}

\newcommand{\E}{\mathrm{E}}

\newcommand{\eqw}{\stackrel{d}{=}}

\newcommand{\tow}{\stackrel{w}{\to}}
\newcommand{\toP}{\stackrel{\mbP}{\to}}

\newcommand{\mbZ}{\mathbb{Z}}

\definecolor{hcolor}{rgb}{0.0, 0.0, 0.0}
\begin{document}

\title
[On a limit behaviour of a random walk penalised in the lower half-plane]
{On a limit behaviour of a random walk penalised in the lower half-plane}

\author{A. Pilipenko}
\address{Institute of Mathematics, National Academy of Sciences of Ukraine, 3 Tereshchenkivska str., 01601, Kyiv, Ukraine; National Technical University of Ukraine ``Igor Sikorsky Kyiv Polytechnic Institute'', Kyiv, Ukraine}
\email{pilipenko.ay@gmail.com}
\thanks{A.Pilipenko was partially supported by the Alexander von Humboldt Foundation within 
the Research Group Linkage Programme {\it Singular diffusions: analytic and stochastic approaches}, 
 the DFG Project \emph{Stochastic Dynamics with Interfaces} (PA 2123/7-1),
 and 
 the National Research
Foundation of Ukraine (project 2020.02/0014 ``Asymptotic regimes of perturbed random walks:
on the edge of modern and classical probability'').} 

\author{Ben Povar}
\address{National Technical University of Ukraine ``Igor Sikorsky Kyiv Polytechnic Institute'', Department of Physics and Mathematics, 03056, Kyiv, Ukraine, 37, Peremohy ave.}
\email{o.prykhodko@yahoo.com}

\subjclass[2000]{60F17; 60G50}
\keywords{Invariance principle, Reflected Brownian motion}

\begin{abstract}
    We consider a random walk $\tilde S$ which has different increment distributions in positive and 
negative half-planes. In the upper half-plane the increments are 
mean-zero i.i.d. with  finite variance.
In the lower half-plane we consider two cases: increments are   positive i.i.d. random variables with either
   a slowly
 varying tail or with a finite expectation. For the distributions with a slowly varying tails, 
we show that $\{\frac{1}{\sqrt n} \tilde S(nt)\}$ has no weak limit in $\De$; alternatively, 
   the weak limit   is a reflected
 Brownian motion. 
\end{abstract}

\maketitle
\section{Introduction and main results}\label{Introduction}

Let random variables $\{\xi_n\}_{n\geq1}$ be i.i.d. with a generic copy $\xi$ and $\mbE \xi = 0, \ \mbE \xi^2 = \sigma^2$; $\{\eta_n\}_{n\geq1}$ are positive i.i.d., independent from the previous ones and with a generic copy $\eta$. 
We consider a random walk $\tilde S$:
\be \label{tildeSOne}
\tilde S(0) = 0, \ \tilde S(n+1) = \begin{cases}
    \tilde S(n) + \xi_{n + 1} \ \text{if} \ \tilde S(n) > 0, \\
    \tilde S(n) + \eta_{n+1} \ \text{if} \ \tilde S(n) \leq 0.
\end{cases} 
\ee

\begin{figure}
\includegraphics[scale=0.85]{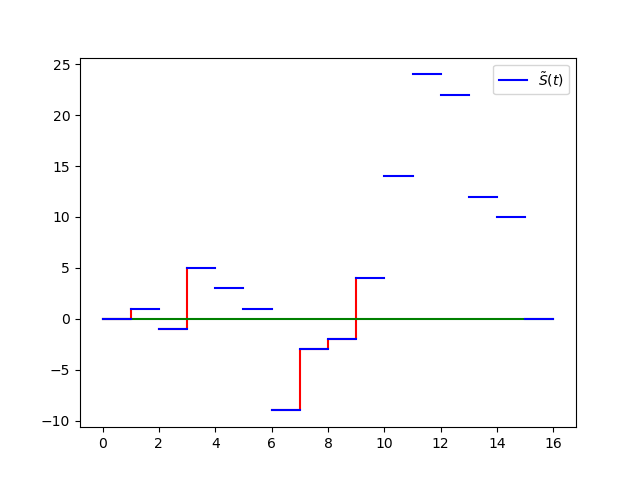}
\caption{Sample path of $\tilde S$ (up to interpolation). \rd{Red vertical lines} show when the jump is distributed as $\eta$.}
\label{TildeSFigure}
\end{figure}
    
If $\tilde S(n)\in(0, \infty)$, then the random walk has increments distributed as $\xi$ and on $(-\infty, 0]$ as $\eta$.
Jumps $\{\eta_n\}$   pushes the random walk up from the negative half-plane. This can be interpreted as some non-immediate reflection to the positive half-plane.  Similar approach 
for stochastic differential equations is called the penalization method and 
leads to a reflected diffusion if penalizations below zero are big enough, see for example 
\cite{Menaldi} or \cite[\S 1.4]{PilipenkoSkorohodProblem}.

An interpretation of this model comes from the evaporation of gas from liquid in physics, for example, a model of ballistic evaporation \cite{Johnson2014}. The usual form of speed distribution of a particle in a medium is Maxwell distribution. But in case when there is a medium change (a particle bounces off liquid) several papers have indicated that super- or sub-Maxwell distributions could apply \cite{Johnson2014}, \cite{Hahn2016}, \cite{Kann2016}. 

Another interpretation is the following.
  Consider an online shop with a warehouse. Assume that the warehouse is replenished every day (e.g. under a long-term contract) and also that the expectation of supply equals the expectation of demand. So,
 the increments of the number of goods in the warehouse are mean-zero random variables $\{\xi_k\}$ if the warehouse is non-empty. 
 In case when the virtual amount of goods in stock is negative, i.e.,  the warehouse is empty and we
 have unsatisfied orders, the shop seeks 
to refill the warehouse by buying batches of goods until it has a positive balance. 
 In our model this refilling corresponds to the sequence  $\{\eta_k\}$. If we consider a regular shop or a traditional queueing
model, then the unsatisfied orders would be discarded and the walk couldn't overjump a zero-level.
While the shortage is not overcome (which may take some time), the shop does not sell any goods. 

Our goal is to analyse the limit of the scaled sequence of the processes $\Big\{\frac{1}{\sqrt n}\tilde S(nt), \ t\in[0,1]\Big\}_{n\geq 1}$ as $n\to\infty$.
 Depending on the distribution of $\eta$ we want to distinguish two cases: 
\begin{enumerate}
    \item $\eta$ has a finite expectation:
    \[\mbE \eta < \infty.\]
    \item the distribution of $\eta$ has a slowly varying tail
    \[ \mbP (\eta > x) \sim l(x), \ l \in R_0, \]
where $R_0$ a set of slowly varying functions (check definition 1.4.2 in \cite{BGH87}).
\end{enumerate}

\begin{remark}\label{AlphaStableCaseRemark}
    Another interesting case is when the tail of $\eta$ is regularly varying, i.e., $\eta$ belongs to the domain of attraction of an $\alpha$-stable distribution. This case was studied in \cite{PilipenkoPrykhodko14} with the restriction on $\xi$ to be in $\mbZ$ and greater than or equal to $-1$. We conjecture that the limit process should be the same as in \cite{PilipenkoPrykhodko14}, although we postpone the proof for the future paper. 
\end{remark}

For  the  case  $\E\eta<\infty$ we assume that $\tilde S$ is linearly interpolated  for $t \geq 0$:  
\[ \tilde S(t) = \tilde S(\floor*{t}) + (t - \floor*{t}) \tilde S(\ceil*{t}). \]
\begin{theorem}\label{ThmFiniteExpectation}
    A sequence of processes 
    \[ \p[\Big]{\frac{\tilde S(n t)}{\sqrt n}, \ t\in [0, 1]}_{n \geq 1} \]
    converges weakly in $\Ce$ to a reflected Brownian motion:
    \[W_{\text{reflected}}(t) := W(t) - \min_{s \leq t} W(s).\]
\end{theorem}
\begin{remark}
    It is well-known that $W_{\text{reflected}}$ has the same distribution as the absolute value of a Brownian motion $|W|$, see for example Theorem 1.3.2 in \cite{PilipenkoSkorohodProblem}.
\end{remark}

For the second case we assume $\tilde S$ has a flat-right interpolation for $t\geq 0$:  
\[ \tilde S(t) = \tilde S(\floor*{t}). \]
\begin{theorem}\label{ThmLTails}
    For a sequence of processes 
    \[ \p[\Big]{\frac{\tilde S(n t)}{\sqrt n}, \ t\in[0,1]}_{n \geq 1} \]
    there is no weak limit in $\De$. Moreover for any $t > 0$
    \be \label{MaxTildeS} \max_{s\leq t} \frac{\tilde S(n s)}{\sqrt n} \toP \infty, \ \ninf. \ee
\end{theorem}

Our results show that intuitively understandable things happen: when $\eta$ has an expectation, the process behaves like a Brownian motion with a reflection; otherwise when $\eta$ has a slowly varying tail, the process blows up at the beginning. 

 Note that model \eqref{tildeSOne} is equivalent (up to recounting of $\xi_i$ and $\eta_i$) to the following 
\be \label{tildeS}
\tilde S(0) = 0, \ \tilde S(n) = \sum_{i=1}^{n-T(n)} \xi_i + \sum_{i=1}^{T(n)}\eta_i, \ n \geq 1, 
\ee
where $T(n)$ is the number of visits to $(-\infty, 0]$ before the time $n$:
\be \label{eq0}
 T(n) = \#\{ k < n : \tilde S(k) \leq 0\}, \ n\geq 1. 
 \ee
 We work only with this representation in the paper further.

Denote
\be S_\xi(n) = \sum_{i=1}^n\xi_i, \ S_\eta(n) = \sum_{i=1}^n\eta_i. \ee
So
\be\label{eq2}
\tilde S(n) = S_\xi(n - T(n)) + S_\eta(T(n)).
\ee

We set by definition $S_\xi(t):= S_\xi(\floor*{t}), \ S_\eta(t):= S_\eta(\floor*{t}), T(t):= T(\floor*{t})$ for all $t\geq0.$

Although the cases we discuss are very different, we tried to pick out common traits in
 Section \ref{SharedPart}. 

In Section \ref{FirstPart} we prove Theorem \ref{ThmFiniteExpectation}. We aim to show that $\frac{1}{\sqrt n}S_\eta(T(nt))$ converges to $- \min_{s \leq t} W(s)$ and do so by showing that the latter lies neatly between $\frac{1}{\sqrt n}S_\eta(T(n t) - 1)$ and $\frac{1}{\sqrt n}S_\eta(T(n t)) + \frac{1}{\sqrt n}\max_{i \leq nt} |\xi_i|$. 

The proof of the second case can be found in Section \ref{SecondPart} and is based on the 
observation that the overshoots above level $n$ of the random walk $S_\eta$ are much larger than $n$. This is a known property, proved by Rogozin, see for example Theorem 8.8.2  
in \cite{BGH87}. Thus we need to show that this overshoot happens before $T(n)$ because only 
in this way it could affect $\tilde S$. 

\section{Proof: shared part}\label{SharedPart}
Let $m$ be a (sign flipped) running minimum of $S_\xi$:
\be\label{eq1}
m(n) = - \min_{k \leq n}  S_\xi(k). 
\ee



\begin{lemma}\label{twofold-ineq-lemma}
    For each $n\geq 1$
    \be \label{twofold-ineq}
    S_\eta(T(n) - 1) \leq m(n - T(n)) < S_\eta(T(n)) + \max_{i \leq n} |\xi_i|
    \ee
\end{lemma}
\begin{proof}
    The left inequality is trivial when $T(n)=1$. Assume that $T(n) > 1$ and that for some $n \geq 1$
    \be\label{Assumption} S_\eta(T(n) - 1) > m(n-T(n)). \ee
    Take $k < n$, such that $T(n) = T(k + 1)$ and $T(k) + 1 = T(n)$. This is always possible since $T(n)>1$. 
    Thus
    \[\tilde S(k) = S_\xi(k - T(k)) + S_\eta(T(k))> -m(k-T(k)) + S_\eta(T(n) - 1). \]
    From \eqref{Assumption} we infer that $\tilde S(k) > 0$. Therefore $T(k) = T(k + 1)$ which contradicts with our choice of $k$.
    
    As for the right hand side inequality, observe that 
    $\tilde S$ jumps downwards at $k \geq 1$ only when $\tilde S(k-1) > 0$ and $S_\xi(k - T(k))$ jumps downwards. Thus for every $k \geq 1$
    \[ -\max_{i \leq k - T(k)} |\xi_i| \leq \tilde S(k) - \tilde S(k-1) < \tilde S(k) = S_\xi(k - T(k)) + S_\eta(T(k)). \]   
    Hence 
    \[ S_\eta(T(k)) + \max_{i \leq k - T(k)} |\xi_i| \geq -S_\xi(k - T(k)). \] 
    Taking maximum over $k \leq n$ of both sides of the last inequality we get 
    \[ S_\eta(T(n)) + \max_{i \leq n} |\xi_i| \geq -\min_{k \leq n} S_\xi(k - T(k)) = -\min_{k \leq n - T(n)} S_\xi(k) = m(n - T(n)). \] 
\end{proof}

By the Donsker theorem 
\[ \frac{1}{\sqrt n}S_\xi(n \cdot) \tow W(\cdot), \ \ninf, \]
meaning a weak convergence in $\Ce$. By the Skorokhod Representation theorem (Theorem 6.7 in \cite{Billingsley77})  we can construct a probability space which supports random elements $\hat S_\xi^{(n)}$ and $\hat W$ such that
\[\hat S_\xi^{(n)} \eqw S_\xi, \ \hat W \eqw W,  \ n\geq 1\]
and the uniform convergence holds
\be \label{invariance} \frac{1}{\sqrt n}\hat S_\xi^{(n)}(n t) \rightrightarrows \hat W(t) \ \text{as} \ \ninf \ \as \ee
for $t \in[0,1]$. 

To emphasize that we work on this new space we will write a $\hat{\text{hat}}$
 whenever we use random variables dependent on $\hat S_\xi^{(n)}$. 
Without loss of generality we may assume that this new probability space is reach enough and contains
a copy of the sequence  $\{\eta_k\}_{k\geq 1}$. We leave the same notation and assume that
the original sequence $\{\eta_k\}_{k\geq 1}$ belongs to  this probability space.
Using sequences $\{ \hat S_\xi^{(n)}(k)\}$ and $\{S_\eta(k)\}$ 
we construct copies   $\{\hat T^{(n)}(k)\}, \ \{ \hat{\tilde S}^{(n)}(k)\},
\ \{ \hat{m}^{(n)}(k)\}$ of $\{ T(k)\}, \ \{  {\tilde S}(k)\},  \{m (k)\}$ similarly to formulas \eqref{eq0}, \eqref{eq1},
 \eqref{eq2}.
 Note that in this constructions sequences
 $\{\hat S_\xi^{(n)}(k)\}$ depend  on $n$ but  $\{S_\eta(k)\}$
 does not. Often we will leave out indexing by $n$, that is $\hat S_\xi^{(n)}$ becomes $\hat S_\xi$. 
 
It follows from inequality \eqref{twofold-ineq} that 
\be \label{SetaM} S_\eta(\hat T(n t) - 1) \leq \hat m(n t)\ \as \ee
for all $t \geq 0$.

Divide both sides by $\sqrt n$. By \eqref{invariance} we have 
\be \label{right-conv}
    \forall \ve_1 > 0 \ \exists n_1 \ \forall n > n_1 \quad
    \frac{1}{\sqrt n} S_\eta(\hat T(n t) - 1)  \leq -\min_{s \leq t} \hat W(s) + \ve_1\ \as
\ee
Note that $\lim_{k\to\infty}\hat T(k) =+\infty$ a.s.
because $\liminf_{k\to\infty}S_\xi(k)=-\infty$
a.s. Thus 
\be \label{right-conv-2}
    \forall \ve_1 > 0 \ \exists n'_1 \ \forall n > n'_1 \quad
    \frac{\hat T(nt) - 1}{\sqrt n}  \leq \frac{\hat T(nt) - 1}{S_\eta(\hat T(n t) - 1)} (-\min_{s \leq t}\hat W(s) + \ve_1) \ \as
\ee

The strong law of the large numbers for $S_\eta$  implies 
\be\label{pre_l2} 
\limsup_{\ninf} \frac{\hat T(n t)}{\sqrt n} \leq \frac{-\min_{s \leq t}\hat W(s)}{\mbE\eta}  \ \as
\ee 
When $\mbE \eta = \infty$, the right hand side is $0$.
 
\begin{corollary}\label{TConv0}
    \[ \sup_{0 \leq t \leq 1} \frac{T(n t)}{n} \toP 0, \ \ninf. \]
\end{corollary}
\begin{proof}
    Since $\hat T$ is non-decreasing we have
    \[\sup_{0 \leq t \leq 1}\frac{\hat T(n t)}{n} \leq \frac{\hat T(n)}{n}.\]
    Recall that $\hat T \eqw T$, hence the corollary follows from \eqref{pre_l2}.
\end{proof}

\begin{lemma}\label{minW}
    The uniform convergence on $[0,1]$ holds as $\ninf$
    \[\frac{\hat m(n t -\hat T(n t))}{\sqrt n} \rightrightarrows -\min_{s \leq t}\hat W(t) \ \as \]
\end{lemma}
\begin{proof}
Equations \eqref{invariance} and \eqref{pre_l2}
yield the uniform convergences on $[0,1]$
     \[ 
\frac{1}{\sqrt n} \hat m (n t) =\frac{1}{\sqrt n} \hat m^{(n)} (n t)  \rightrightarrows 
-\min_{s \leq t}\hat W(t),\ \ n\to\infty,
\]
and
\[ 
    t - \frac{\hat T(n t)}{n}=t - \frac{\hat T^{(n)}(n t)}{n} \rightrightarrows t,\ \ n\to\infty,
\]
  with probability 1. 

Since $ t - \frac{\hat T(n t)}{n} \in [0,1]$ and the limit functions are continuous,
 their 
compositions a.s.  uniformly converge to the composition of limits. 
\end{proof}
\begin{remark}
In Lemma \ref{minW} we may use  both  linear and
 piecewise constant approximation of the corresponding processes.
\end{remark}
\section{Proof of Theorem \ref{ThmFiniteExpectation}}\label{FirstPart}
It suffices to prove uniform convergence over $t \in [0, 1]$ as $\ninf$ 
    \be\label{SEtaMinW} 
    \frac{1}{\sqrt n}S_\eta(\hat T(n t)) \rightrightarrows -\min_{s \leq t} \hat  W(s) \ \as \ee
    The other part of the statement of Theorem \ref{ThmFiniteExpectation} is provided by \eqref{invariance}.

From Lemma \ref{twofold-ineq-lemma} and Lemma \ref{minW} we infer that  
\begin{align}
    &\limsup_{\ninf} \frac{1}{\sqrt n} S_\eta(\hat T(n t) - 1)  \leq -\min_{s \leq t}\hat  W(s) \ \as \\
    &\liminf_{\ninf} \p[\Bigg]{ \frac{1}{\sqrt n} S_\eta(\hat T(n t)) + \frac{\max_{i \leq n}  |\hat\xi_i|}{\sqrt n} }\geq -\min_{s \leq t} \hat W(s) \ \as
\end{align}

To prove \eqref{SEtaMinW} it suffices to show 
    \be \label{MaxXi} 
    \frac{\max_{i \leq n}  |\hat \xi_i|}{\sqrt n}\toP 0, \ \ninf.\ee
    and 
    \be \label{etaT}  
    \frac{1}{\sqrt n} |S_\eta(\hat T(n t)) -  S_\eta(\hat T(n t) - 1)| = \frac{1}{\sqrt n}\eta_{\hat T(n t)}\toP 0, \ \ninf, \ee

    Observe that for a given $\delta > 0$ due to \eqref{pre_l2} it is always possible to choose $K > 0$ such that 
    \[ \mbP \p{\hat T(n) \leq K \sqrt n} > 1 - \delta. \]
    Hence both \eqref{MaxXi} and \eqref{etaT} follow from Theorem 6.2.1 in \cite{Gut2013}.



\section{Proof of Theorem \ref{ThmLTails}}\label{SecondPart}
It suffices to prove \eqref{MaxTildeS}. Set
\[N_\eta(x) = \inf\{k : S_\eta(k) > x\}, \ x \geq 0.\]
Let $c > 0$, $K > 0$, $0<\ve<K$ and $0<t\leq1$ be arbitrary real numbers, then 
\begin{align}\label{thin_road}
\begin{split}
\event{\max_{i \leq n t} \frac{\tilde S(i)}{\sqrt n} \geq c} &= \event{\max_{i \leq n t} \p[\Big]{S_\xi(i - T(i)) + S_\eta(T(i))} \geq c \sqrt n}  \\
&\supset \event{\min_{i \leq n} S_\xi(i) \geq -K \sqrt n} \cap 
\event{S_\eta(T(nt)) > (K + c) \sqrt n }  \\
&= \event{\min_{i \leq n}S_\xi(i)  \geq -K \sqrt n} \cap 
\event{N_\eta((K+c) \sqrt n) \leq T(n t)} \\
&\supset \event{\min_{i \leq n}S_\xi(i)  \geq -K \sqrt n} \cap \event{N_\eta(\ve \sqrt n) \leq T(n t)} \\ 
&\quad\quad\quad\quad\quad\quad\quad\quad\quad\quad\quad \cap \event{\eta_{N_\eta(\ve \sqrt n)} > (K + c) \sqrt n}.
\end{split}
\end{align}
For the second inclusion observe that if 
\[\eta_{N_\eta(\ve \sqrt n)} >  (K + c) \sqrt n, \]
then automatically 
\[ N_\eta( (K + c) \sqrt n) = N_\eta(\ve \sqrt n). \]

\begin{lemma}\label{lem:1}
Let $j, l, a > 0$ be real numbers such that $j + l < n$, then
\[ \event{\min_{i \leq j}S_\xi(i) \leq -a} \cap \event{T(n) < l} \subset \event{N_\eta(a) \leq T(j + l)}. \]
\end{lemma}
\begin{proof} 
    Suppose $\min_{i \leq j}S_\xi(i) = -A \leq -a$ and it is firstly attained at $i = i^* \leq j$. Denote by $u$ the smallest solution of 
    $$u - T(u) = i^*.$$
    
    Let $k$ be such that  
    \[\tilde S(u) \leq \tilde S(u+1) \leq\dots\leq \tilde S(u+k-1)\leq 0 < \tilde S(u+k),\] 
    meaning that $\tilde S$ needs $k$ jumps to become positive. Observe that $T(u) + k = T(u + k)$.

    Assume $T(n) < l$. Then $k < l$. 
    Thus 
    \[u + k < j + l < n.\] 
    
    
    As $k$ is such that $\tilde S(u + k) > 0$, then from definition of $\tilde S$ we deduce that $S_\eta(T(u+k)) > A$ and so 
    \[N_\eta(a) \leq N_\eta(A) \leq T(u + k) = T(i^* + T(u) + k) = T(i^* + T(u + k)) \leq T(j + l).\]
\end{proof}

Apply Lemma \ref{lem:1} with $j=l=\frac{nt}{2}$ and $a = \ve \sqrt n$ to further inclusion \eqref{thin_road}
\be
\begin{split}\label{broader_road}
&\event{\max_{i \leq n t} \frac{\tilde S(i)}{\sqrt n} \geq c}  \\ 
\supset \event{\min_{i \leq n}S_\xi(i)  \geq -K \sqrt n} &\cap \event{N_\eta(\ve \sqrt n) \geq T(n t)} \cap \event{\eta_{N_\eta(\ve \sqrt n)} > (K + c) \sqrt n}\\
\supset \event{\min_{i \leq n}S_\xi(i)  \geq -K \sqrt n} &\cap \event{\min_{i \leq n t / 2} S_\xi(i) \leq -\ve \sqrt n} \cap \event{T(n) < \frac{nt}{2}} \\
& \cap \event{\eta_{N_\eta(\ve \sqrt n)} > (K + c) \sqrt n}.\\
\end{split}
\ee

Hence 
\begin{align}  \label{toproof}
\begin{split}  
    \Prob{\max_{i \leq n t} \frac{\tilde S(i)}{\sqrt n}\geq c} &\geq  \Prob{\text{r.h.s. of \eqref{broader_road}}} \\
    &= 1 - \Prob{\text{complement to the r.h.s. of \eqref{broader_road}}} \\
    &\geq 1 - \Prob{\min_{i \leq n}S_\xi(i)  \leq -K \sqrt n} - \Prob{\min_{i \leq n t / 2} S_\xi(i) \geq -\ve \sqrt n} \\
    &- \Prob{T(n) > \frac{nt}{2}} - \Prob{\eta_{N_\eta(\ve \sqrt n)} < (K + c) \sqrt n}.
\end{split}
\end{align}

Let $\delta > 0$ be an arbitrary small number. We proceed by proving that every negative term in the r.h.s. of the last inequality is greater than $-\delta$ for $n$ big enough.

Choose $K > 0$ and $0 < \ve < K$ so that 
\[ \Prob{|N_\eta(0, 1)| > K} < \delta, \quad \Prob{|N_\eta(0, 1)| < \ve \sqrt \frac{2}{t}} < \delta. \]
Since 
\[ -\min_{i \leq n t} \frac{S_\xi(i)}{\sqrt n} \weakto |N_\eta(0, t)|, \ \ninf, \]
we get 
\[ \limsup_{\ninf} \p[\Bigg]{ \Prob{\min_{i \leq n}S_\xi(i)  \leq -K \sqrt n} + \Prob{\min_{i \leq n t / 2} S_\xi(i) \leq -\ve \sqrt n}} \leq 2\delta. \]

 It follows from Theorem 8.8.2 in \cite{BGH87} that
\be\label{BGHTheorem882} \frac{\eta_{N_\eta(\ve \sqrt n)}}{\sqrt n} \toP \infty, \ \ninf. \ee
So the last term in the right hand side of \eqref{toproof} converges to 0.

Corollary \ref{TConv0} concludes the proof.

\end{document}